\documentclass{proc-l}

\usepackage{mathrsfs}
\usepackage{amssymb,amsmath,amsthm,color}
\usepackage{graphicx}
\usepackage{hyperref}
\usepackage{url}
\usepackage{setspace}
\usepackage{comment}

\DeclareMathOperator{\supp}{supp}
\DeclareMathOperator{\sign}{sign}

\def\cc{{\mathcal C}}

\def\ff{{\mathcal F}}

\def\lll{{\mathcal L}}

\def\ss{{\mathcal S}}

\def\ffi{\varphi}
\def\eps{\varepsilon}
\def\dst{\displaystyle}

\renewcommand{\Im}{\mathrm{Im}\,}

\def\C{{\mathbb{C}}}

\def\N{{\mathbb{N}}}

\def\R{{\mathbb{R}}}

\def\Z{{\mathbb{Z}}}

\def\heis{{\mathbb{H}^n}}
\newcommand{\norm}[1]{{\left\|{#1}\right\|}}
\newcommand{\ent}[1]{{\left[{#1}\right]}}

\newcommand{\scal}[1]{{\left\langle{#1}\right\rangle}}

\newtheorem{theorem}{Theorem}[section]

\newtheorem{proof of lemma}[theorem]{Proof of Lemma}
\newtheorem{proposition}[theorem]{Proposition}

\theoremstyle{definition}
\newtheorem{definition}[theorem]{Definition}
\newtheorem{example}[theorem]{Example}

\newtheorem{remark}[theorem]{Remark}

\numberwithin{equation}{section}

\makeatletter
\@namedef{subjclassname@2020}{\textup{2020} Mathematics Subject Classification}
\makeatother

\begin{document}

\title[Schr{\"o}dinger equation on the Heisenberg group]
{Uncertainty Principle for solutions of the Schr{\"o}dinger equation on the Heisenberg group}

\author{Phiippe Jaming \& Somnath Ghosh}
%
\address{Department of Mathematics, Indian Institute of Science, Bangalore 560012, India.}
\email{somnath.g.math@gmail.com}

\address{Univ. Bordeaux, CNRS, Bordeaux INP, IMB, UMR 5251, F-33400, Talence, France.}
\email{Philippe.Jaming@math.u-bordeaux.fr}
\thanks{The first author was partially supported by ANR Grant \#ANR-24-CE40-5470 CHAT. This research was partially
undertaken during the second author's visit to IMB-Bordeaux, he would like to thank IMB for hospitality and finantal contribution to his visit.}


\date{\today}


\begin{abstract}
The aim of this paper is two prove two versions of the Dynamical Uncertainty Principle
for the Schr\"odinger equation $i\partial_s u=\mathcal{L}u+Vu$, $u(s=0)=u_0$ where
$\mathcal{L}$ is the sub-Laplacian on the Heisenberg group.

We show two results of this type. For the first one, the potential $V=0$, we establish a dynamical version of 
Amrein-Berthier-Benedicks's Uncertainty Principle that shows that if $u_0$ and $u_1=u(s=1)$ have both small 
support then $u=0$. 
For the second result, we add some potential to the equation and we obtain a dynamical version of the Paley-Wiener
theorem in the spirit of the result of Kenig, Ponce \& Vega \cite{KPV}.
Both results are obtained by suitably transfering results from the Euclidean setting.

We also establish some limitations to Dynamical Uncertainty Principles.
\end{abstract}

\maketitle

\section{Introduction}

The aim of this paper is to prove several Dynamical Uncertainty Principles
for the Schr\"odinger operator associated to the sub-Laplacian of the Heisenberg group.

Let us recall that in Euclidean harmonic analysis, an Uncertainty Principle
is a statement that {\em a function $f$ and its Fourier transform $\widehat{f}$ can not be simultaneously localised}.
There are numerous statements of this form, depending on how localization of functions is measured.
The most famous results of this form are the Heisenberg Uncertainty Principle when localization
is measured in terms of dispersion, and Hardy's Uncertainty Principle when localization
is measured in terms of fast decay. Another family of Uncertainty Principles measures
localization in terms of smallness of the support. For instance, a result by Amrein-Berthier \cite{AB} 
and Benedicks \cite{B} states that a function and its Fourier transform can not both have support of finite measure.
This leads to the introduction of the concept of {\em annihilating pairs} (see Definition 2.1 below and the examples following it): a pair of sets $(S,\Sigma)$ such that a function $f$ supported in $S$ with Fourier transform $\widehat{f}$ supported in $\Sigma$
is necessarily $f=0$.
We refer to the surveys \cite{FS,BD,FBM,M} or the book \cite{HJ} 
for more information on the subject. 

\smallskip

As the Uncertainty Principle is a central phenomena in Euclidean harmonic analysis, there has been a lot
of activity aiming at its extension to other settings, among which Lie groups, and in particular
to nilpotent Lie groups. There has been a lot of activity around Hardy's Uncertainty Principle
({\it see} \cite{T2} and references therein) or on the Theorem of Amrein-Berthier-Benedicks ({\it see e.g.} \cite{NR,GS}).
Here a new difficulty occurs: while in the Euclidean setting, the function and its Fourier
transform are both defined on $\R^d$, this is no longer the case
for the Fourier transform on Lie groups.

\smallskip

One way to overcome this, is to take the dynamical reformulation of the Uncertainty Principle.
The idea is rather simple: the solution of the Schr\"odinger equation
$$
\left\{\begin{matrix} i\partial_tu(x,t)+\Delta_xu(x,t)=0\\
u(x,0)=u_0(x)\end{matrix}\right.
$$
is given by
$$
u(x,t)=\frac{e^{-i\frac{\pi}{4}d\sign(t)}}{(4\pi|t|)^{d/2}}e^{i|x|^2/4t}
\int_{\R^d}e^{i|y|^2/4t}u_0(y)e^{-i\scal{x,y}/2t}\,\mbox{d}y.
$$
Thus $u_1(x)=u(x,1)$ is, up to constant modulus factors that do not affect concentration, essentially the Fourier
transform of $u_0$. One may then restate any Uncertainty Principle for the pair $(f,\widehat{f})$
in terms of an Uncertainty Principle for $(u_0,u_1)$. Such a result is now called a Dynamical Uncertainty Principle.
For instance, we have the following
\begin{enumerate}
\item {\bf Hardy's Uncertainty Principle:}\\
{\em Assume that $u_0(x)=O(e^{-|x|^2/\alpha^2})$ and that 
$u(x,T)=O(e^{-|x|^2/\beta^2})$. If $T/\alpha\beta>1/4$, then $u_0=0$ and if $T/\alpha\beta=1/4$,
then $u_0$ is a multiple of $e^{-(1/\alpha^2+i/4T)|x|^2}$.}

\item {\bf Amrein-Berthier-Benedicks:}\\
{\em Assume that $u_0$ and $u(x,T)$ both have finite support, then
$u_0=0$.}
\end{enumerate}

This simple observation, is the starting point of two fruitful lines of research. 

On one hand,
in a seminal series of papers, Escauriaza, Kenig, Ponce and Vega \cite{EKPV1,EKPV2,EKPV3} have extended
Hardy's Theorem and proved Dynamical Uncertainty Principles associated to the Schr\"odinger Equation 
with potential $V$: $i\partial_tu=\Delta_xu+Vu$ with various conditions on $V$
(as well as for other dispersive equations). It is noteworthy to say that in this case, the result does
not follow from its non-dynamical counterpart and requires a different strategy of proof
using only real variable techniques. It turns out that this strategy is so powerful that, in the absence of 
potential, one may even recover Hardy's original result \cite{CEKPV}.

On the other hand, this formulation is amenable to extension to other settings in which there is a natural notion
of Laplacian (e.g. on graphs or Lie groups), as the variable space does not change with time. The first
result in that direction is due to S. Chanillo \cite{Ch}. The dynamical
version of Hardy's Uncertainty Principle (without potential), on the $H$-type groups has been  first investigated by 
S. Ben Sa\"{\i}d, S. Thangavelu and V.\,N. Dogga \cite{BTD}
and for general step 2 nilpotent Lie groups by J. Ludwig and D. M\"uller \cite{LM}.  
For the Dynamical Uncertainty Principle with potential, see A. Fernandez-Bertolin {\it et al} \cite{FJP}.

Our aim here is precisely to pursue this direction of research for the notion of annihilating pairs
when dealing with the Schr\"odinger equation associated to the sub-Laplacian on the Heisenberg group.
Before describing our results, we need to recall some notions on the Heisenberg group
as can be found {\it e.g.} in \cite{F}.

\smallskip

The $(2n + 1)$-dimensional Heisenberg group, denoted by $\heis$, is $\R^n\times\R^n\times\R$
endowed with the group law
$$
(x,y, t)(x',y', t') = \left(x+x',y+y',t+t'-\frac{1}{2}\sum_{j=1}^n(x_iy_i^\prime-y_ix_i^\prime)\right).
$$
Under this multiplication $\heis$ is a nilpotent unimodular Lie group, the Haar measure
being the Lebesgue measure $\mbox{d}x\,\mbox{d}y\,\mbox{d}t$.
It will be convenient to write $z=(x,y)\in\R^{2n}$, $(z,t)=(x,y,t)\in\heis$
and the Haar measure is then denoted by $\mbox{d}z\,\mbox{d}t$ on $\R^{2n}\times\R$.

It is sometimes convenient to identify $\R^{2n}$ with $\C^n$ writing $z=x+iy$ and then
$\C^n\times\R\simeq\R^{2n}\times\R$.
In this case, the group law on $\heis\simeq\C^n\times\R$ is given by 
$$
(z,t)(z',t')=\bigl(z+z',t+t'+\frac{1}{2}\Im(z\cdot\overline{w})\bigr)
$$
and the Haar measure is still denoted by $\mbox{d}z\,\mbox{d}t$.

The Lie algebra associated to $\heis$ is generated by
the vector fields
$$
X_j:=\frac{\partial}{\partial x_j}+\frac{1}{2}y_j\frac{\partial}{\partial t}
\quad,\quad
Y_j:=\frac{\partial}{\partial y_j}-\frac{1}{2}x_j\frac{\partial}{\partial t}
\quad j=1,2,\ldots,n
$$
and $T=\dfrac{\partial}{\partial t}$.
The sub-Laplacian for the Heisenberg group $\heis$, is given by
\[
\mathcal{L}=\sum_{j=1}^n X_j^2+Y_j^2=
\Delta_{\mathbb R^{2n}}+\frac{|z|^2}{4}\frac{\partial^2}{\partial t^2}-\sum_{j=i}^n
\left(x_j\frac{\partial}{\partial y_j}-y_j\frac{\partial}{\partial x_{j}}\right)\frac{\partial}{\partial t}.
\]
It is well-known that $\lll$ is hypo-elliptic, self-adjoint and non-negative. Further $e^{i\lll}$ generates
a unitary semi-group on $L^2(\heis)$ and the solution of the free Schr{\"o}dinger equation for $\lll$
\begin{equation}\label{exp971}
i\partial_s u(z,t,s)+\mathcal{L}u(z,t,s)=0
\end{equation}
with initial condition $u(z,t,0)=u_0(z,t)$, $u_0\in L^2(\mathbb H^n)$ is given by
$u(\cdot,\cdot,s)=e^{is\lll}u_0$.

The following Dynamical Uncertainty Principle of Hardy type
associated to $\mathcal{L}$ has been shown in \cite{BTD}:

\begin{theorem}[Ben Sa\"id \& Thangavelu \& Dogga]
\label{th:BST}
Let $u_0\in L^2(\heis)$ and $u$ be the solution of \eqref{exp971} with initial condition $u(z,t,0)=u_0(z,t)$.
Assume that, for some time $T>0$,
$$
|u_0(z,t)|=O(e^{-|x|^2/4a^2-\delta|t|})\quad\mbox{and}\quad
|u(z,t,T)|=O(e^{-|x|^2/4b^2-\delta|t|})
$$
for some $\delta>0$ and $ab<T$, then $u=0$.
\end{theorem}

The result has been extended to all step 2 nilpotent Lie groups in \cite{LM} and to Schr\"odinger
equation with certain potentials on H-type groups in \cite{FJP}.
Among the tasks we tackle in this paper is to show that the extra factor $e^{-\delta|t|}$
is essential, though $\delta$ can be arbitrarily small. We will then
show that this is no longer the case when considering a dynamical version
of Amrein-Berthier-Benedicks' Uncertainty Principle on the Heisenberg group
where the central variable plays no role:

\begin{theorem}
\label{th:ben}
Let $S,\Sigma\subset\R^{2n}$ be two sets of finite measure.
Let $u_0\in L^2(\heis)$ and $u$ be the solution of \eqref{exp971} with initial condition $u(z,t,0)=u_0(z,t)$.
Assume that $\supp u_0\subset S\times\R$ and that, for some time $s>0$
$\supp u(\cdot,\cdot,s)\subset \Sigma\times\R$. Then $u=0$
\end{theorem}

Our actual result is more general and provides a way of transferring other
results about some annihilating pairs from the Fourier transform to
Dynamical Uncertainty Principles associated to the sub-Laplacian of the Heisenberg group.
However, we will show that a more quantitative results (so-called strong annihilating pairs)
do not transfer to this more general setting.

We then investigate the influence of the potential. Even in the Euclidean setting, not much is known,
excepted for a result by Kenig, Ponce and Vega \cite{KPV}. We will here show that this result
can also be transferred to the Heisenberg group. 
This result is a form of Paley-Wiener theorem:

\begin{theorem}\label{th921}
Let $u\in C\left([0,1]:L^2(\mathbb H^n)\right)$ be a solution of 
\begin{equation}\label{exp921}
i\partial_s u(z,t,s)+\mathcal{L}u(z,t,s)+V(z,s)u(z,t,s)=0,
\end{equation}
where $(z,t)\in\mathbb H^n,$ $s\in[0,1]$ and $V:\mathbb H^n\times [0,1]\rightarrow\mathbb C$ is independent
of the central variable $t$.

Assume further that the potential $V$
is bounded, i.e., $\|V\|_{L^\infty (\mathbb C^n\times [0,1])}<\infty$ and
\begin{equation}\label{exp922}
\lim_{\rho\rightarrow\infty}\int_0^1\|V(\cdot,s)\|_{L^\infty(\mathbb C^n\setminus B_\rho)}\,\mathrm{d}s=0.
\end{equation}
Suppose
\begin{eqnarray}
&&\sup_{0\leq s\leq 1}\int_{\mathbb H^n}|u(z,t,s)|^2 \,\mathrm{d}z\,\mathrm{d}t<\infty, \label{exp923} \\
&&\int_{\mathbb H^n} e^{2a_0 x_1}|u(z,t,0)|^2 \,\mathrm{d}z \,\mathrm{d}t<\infty \quad \text{for some } a_0>0,\label{exp924}
\end{eqnarray}
and
\begin{eqnarray}
&&\begin{aligned}
&\supp\left(u(\cdot,\cdot,1)\right)\subseteq \left\{(z,t)\in\mathbb H^n:|z_1|\leq a_1, |t|\leq a_2\right\}\\
&\mbox{for some }0<a_1,a_2<\infty.
\end{aligned}\label{exp925}
\end{eqnarray}
Then $u=0$.
\end{theorem}

\smallskip

The remaining of this paper is organized as follows. The next section is devoted to the proof of
Theorem \ref{th:ben} with Section 2.2 devoted to establishing some limitations
on the Dynamical Uncertainty Principle on the Heisenberg Group. In Section 3, we prove
Theorem \ref{th921}.

\section{Dynamical Uncertainty Principles associated to $\lll$}

In this section, we will need the following notation.
For a function $f\in L^1(\heis)$, we write
$$
f^\lambda(x,y)=\int_\R f(x,y,t)e^{it\lambda}\,\mbox{d}t
$$
for the partial Fourier transform of $f$ in the central variable $t$.
The convolution of two functions $f,g\in L^1(\heis)$ is defined by
$$
f*g(z,t)=\int_{\heis}f(w,s)g\bigl((w,s)^{-1}(z,t)\bigr)\,\mbox{d}w\,\mbox{d}s
$$
and an easy computation shows that
$$
(f*g)^\lambda(z)=f^\lambda*_\lambda g^\lambda(z):=
\int_{\C^n}f^\lambda(z-w)g^\lambda(w)e^{i\frac{\lambda}{2}\Im(z\cdot\bar w)}\,\mbox{d}w.
$$
This last expression is called the $\lambda$-twisted convolution of $f^\lambda$ and $g^\lambda$.

\subsection{Annihilating pairs}

Let us recall the following classical notion when dealing with Uncertainty Principles \cite{HJ}:

\begin{definition}
Let $S,\Sigma$ be two measurable subsets of $\R^d$. Then 
\begin{enumerate}
\renewcommand{\theenumi}{\roman{enumi}}
\item $(S,\Sigma)$ is a (weak) \emph{annihilating pair} if, 
$$
\supp f\subset S\quad\mbox{and}\quad\supp\widehat{f}\subset\Sigma
$$
implies $f=0$.

\item 
$(S,\Sigma)$ is called a \emph{strong annihilating pair} if there exists $C=C(S,\Sigma)$
such that
\begin{equation}
\label{eq:sa1}
\norm{f}_{L^2(\R^d)}\leq C\bigl(\norm{f}_{L^2(\R^d\setminus S)}+\|\widehat{f}\|_{L^2(\R^d\setminus\Sigma)}\bigr).
\end{equation}
\end{enumerate}
\end{definition}

Note that $(S,\Sigma)$ is (strongly) annihilating if and only if $(\Sigma,S)$ is (strongly) annihilating.

\begin{example}\label{examp1}
Here are some classical examples:

\smallskip

\noindent{\it(i)} If $S$ is compact and $\R^d\setminus\Sigma$ has an accumulation point, then $(S,\Sigma)$
is a weak annihilating pair since $\supp f\subset S$ implies that $\widehat{f}$ is analytic.

\smallskip

\noindent{\it(ii) Logvinenko-Sereda, \cite{LS}.} \\
If $\Sigma$ is compact then $(S,\Sigma)$ is a strong annihilating pair if and only if 
$E=\R^d\setminus S$
is dense in the sense that there exists $a,\gamma>0$ such that, for every $x\in\R^d$, 
$|E\cap x+[-a,a]^d|\geq \gamma |[-a,a]^d|$. 

In this case, an essentially sharp estimate of the constant is due to Kovrijkine \cite{Ko} in dimension $d=1$ 
and the proof can be extended to arbitrary dimension:
if $\Sigma\subset[-c,c]^d$ then $C(S,\Sigma)\leq \left(\dfrac{\kappa}{\gamma}\right)^{\kappa(ac+1)}$
where $\kappa$ is a constant that depends on $d$ only.

\smallskip

\noindent{\it(iii) Amrein-Berthier-Benedicks, \cite{AB,B}.} \\
If $S,\Sigma\subset\R^d$ are two sets of finite measure, then  $(S,\Sigma)$
is a strong annihilating pair.

Further, an essentialy sharp estimate of the constant is due to Nazarov \cite{N} in dimension $d=1$
(see \cite{J} for $d\geq 1$ where sharpness is unknown): $C(S,\Sigma)\leq \kappa e^{\kappa |S||\Sigma|}$
where $ \kappa $ is a constant that depends on $d$ only.

\smallskip

\noindent{\it(iv) Bonami-Demange, \cite{BD}.} \\
A slight elaboration on Benedicks' proof was given in \cite{BD}:
Consider the class of sets
$$
\ff=\{E\subset \R^d\,:\mbox{ for almost every }x\in\R^d,\ (x+\Z^d)\cap E\mbox{ is finite}\}.
$$
The class $\ff$ contains sets of finite measures, as well as $\ff'$, the class of sets
that concentrate on a finite number of lines, where $E\in\ff'$ is defined via the following property:

{\em There is a finite set of unit vectors $e_1,\ldots,e_N\subset\R^d$, and functions $\eta_1,\ldots,\eta_k$ with $\lim_{|x|\to+\infty}\eta_j(x)=0$ such that
if $x\in E$ then $|x-\scal{x,e_k}e_k|\leq \eta_k(x)$ for some $k\in\{1,\ldots,N\}$.}

For instance, for any $C>0$, $\{(x,y)\in\R^2\,:|xy|\leq C\}\in \ff'$.
Then Bonami-Demange proved that if $S,\Sigma\in\ff$, then $(S,\Sigma)$ is a weak annihilating pair.

\smallskip

For each of these pairs, one should notice that if $T$ is an invertible linear transform,
then $(S,T\Sigma)$ and $(TS,\Sigma)$ are still annihilating pairs.

\smallskip

\noindent{\it(v)  Shubin-Vakilian-Wolff, \cite{SVW}.} \\
Define $\rho(x)=\min(1,|x|^{-1})$ and say that $E$ is $\eps$-thin ($0<\eps<1$) if, for every $x\in\R^d$,
$\big|E\cap B\big(x,\rho(x)\big)\big|\leq\eps\big|B\big(x,\rho(x)\big)\big|$. The sets $\{(x,y)\in\R^2\,:|xy|\leq C\}$ are examples of thin sets in this sense.
Then, there is an $0<\eps_0\leq 1$
such that, if $0<\eps<\eps_0$ is small enough
and $S,\Sigma$ are $\eps$-thin, $(S,\Sigma)$ is a strong annihilating pair. 

It is reasonable to conjecture that one can take $\eps_0=1$ so that there is no restriction on $\eps$.
If this conjecture were true, then $(S,T\Sigma)$ and $(TS,\Sigma)$ would still be annihilating pairs
for any invertible linear transform $T$. The current result requires $T$ to have determinant near enough to $1$. 
\end{example}

We can now state our transference result:

\begin{theorem}\label{th922}
Let $S,\Sigma\subset \R^{2n}$ be such that, for every invertible linear transform $T$, $(S,T\Sigma)$
is a weak annihilating pair.

Let $u\in C^1([0,+\infty):L^2(\mathbb H^n))$ be a solution of the free Schr{\"o}dinger equation for the sub-Laplacian of 
the Heisenberg group $\mathbb H^n,$
\begin{equation*}
i\partial_s u(z,t,s)+\mathcal{L}u(z,t,s)=0
\end{equation*}
with initial condition $u(z,t,0)=u_0(z,t)$, $u_0\in L^2(\mathbb H^n)$.
Assume that
$$
\supp u_0\subset S\times \R
\quad\mbox{and}\quad
\supp u(\cdot,\cdot,s_0)\subset \Sigma\times\R
$$
for some $0<s_0<\infty.$ Then $u=0$.
\end{theorem}

The main difficulty here is that one needs to use intricate distribution theory
to write the solution via a kernel and that this kernel is not nice enough to be exploited
({\it see} \cite{BG} for more details about the kernel). To overcome this difficulty
we will take a Fourier transform in the central variable. This changes the Schr\"odinger
equation into another equation that can be solved via an integral kernel and this can be linked
to the Euclidean Fourier transform. The price to pay is that our condition on the support
of $u_0$ is a bit restrictive and does not cover general sets of finite measure in $\heis$.

\begin{proof}
For $s\geq 0$, $z\in\C^n$ and $\lambda\in\R$, let
$$
u^\lambda(z,s)=\int_\R u(z,t,s)e^{i\lambda t}\,\mbox{d}t
$$
be the inverse Fourier transform of $u(z,t,s)$ in the central variable
and $u_0^\lambda$ be the inverse Fourier transform of $u_0$ in the central variable.
Then
\begin{equation}\label{exp972}
u^\lambda(z,s_0)=u_0^\lambda(z) *_\lambda h_{is_0}^\lambda (z)
\end{equation}
for a.e. $\lambda\in\mathbb R\setminus\{0\},$ where 
\begin{equation}\label{exp973}
h_{is_0}^\lambda (z)=(4\pi)^{-n} \left(\dfrac{\lambda}{i \sin(\lambda s_0)}\right)^n
e^{i\frac{\lambda}{4} \cot(\lambda s_0) |z|^2}.
\end{equation}
Note that this formula is valid for $\lambda\in\R\setminus \frac{\pi}{s_0}\Z$.

Next note that $\supp u^\lambda(\cdot,0)\subset S$ and $\supp u^\lambda(\cdot,s_0)\subset \Sigma$.
From \eqref{exp972} and \eqref{exp973}, we have
\begin{multline}
\label{exp974}
u^\lambda(z,s_0)
=\int_{\mathbb C^n} u_0^\lambda(w)h_{is_0}^\lambda (z-w)
e^{-\frac{i\lambda}{2}\text{Im}(z\cdot\bar{w})}\,\mbox{d}w \\
=(4\pi)^{-n} \left(\dfrac{\lambda}{i \sin(\lambda s_0)}\right)^n\int_{\mathbb C^n} u_0^\lambda(w)
e^{i\frac{\lambda}{4} \cot(\lambda s_0) |z-w|^2}e^{-\frac{i\lambda}{2}\text{Im}(z\cdot\bar{w})}\,\mbox{d}w \\
=c_{\lambda,s_0}e^{i\frac{\lambda}{4} \cot(\lambda s_0)|z|^2} \int_{\mathbb C^n}
u_0^\lambda(w) e^{i\frac{\lambda}{4} \cot(\lambda s_0) |w|^2}
e^{-\frac{i\lambda}{2}\left(\cot(\lambda s_0)\text{Re}(z\cdot\bar{w})+\text{Im}(z\cdot\bar{w})\right)}\,\mbox{d}w,
\end{multline}
where $c_{\lambda,s_0}=(4\pi)^{-n} \, \lambda^n/(i \sin(\lambda s_0))^n.$
Now write $z=x+iy,$ $w=\xi+i\eta$ and $J=\begin{pmatrix}
0 & I_n \\
-I_n & 0
\end{pmatrix},$ and observe that
$$
\begin{aligned}
\cot(\lambda s_0)\text{Re}(z\cdot\bar{w})&+\text{Im}(z\cdot\bar{w})\\
=&\cot(\lambda s_0)(\scal{x,\xi} +\scal{y,\eta})+(\scal{y,\xi}-\scal{x,\eta}) \\
=&\scal{(\cot(\lambda s_0)I_{2n}+J)(x,y), (\xi,\eta)}.
\end{aligned}
$$
Further define $\gamma_\lambda$ on $\R^n\times\R^n$ as
$$
\gamma_\lambda(u,v)=e^{i\frac{\lambda}{4} \cot(\lambda s_0)(|u|^2+|v|^2)}
$$
and note that this function takes its values in the set of complex numbers of modulus $1$.
We can now write \eqref{exp974} as
\begin{equation}\label{exp976}
u^\lambda(x,y,s_0)=c_{\lambda,s_0}\gamma_\lambda(x,y)\ff[\gamma_\lambda u_0^\lambda]
\bigl(J_\lambda(x,y)\bigr)
\end{equation}
where $\ff$ is the Euclidean Fourier transform on $\R^{2n}$ and
$J_\lambda=\frac{\lambda}{2}(J+\cot(\lambda s_0)I_{2n}).$

Observe that
$$
\det J_\lambda=\left(\frac{\lambda}{2}\right)^{2n}\bigl(\cot^2(\lambda s_0)+1\bigr)^n\not=0
$$
so that $J_\lambda$ is invertible. Thus,
as $u^\lambda(\cdot,s_0)$ is supported in $\Sigma$,
$\ff[\gamma_\lambda u_0^\lambda]$ is supported in $J_\lambda\Sigma$.
On the other hand, $\gamma_\lambda u_0^\lambda$ is supported in $S$.
As $(S,J_\lambda\Sigma)$ is an annihilating pair, it follows that $\gamma_\lambda u_0^\lambda=0$ for almost every $\lambda$, thus $u_0=0$.
\end{proof}

\begin{remark}
The derivation of Formula of the kernel $h_{is_0}$ can be found {\it e.g.} in \cite{LM} even replacing
$\heis$ by a more general step 2 nilpotent Lie group.
\end{remark}

One would of course like to obtain a quantitative counterpart of this result.
We will show that this is not fully possible below.

Let us show a partial result that will highlight the difficulties.
Take $S,\Sigma\subset\R^{2n}$ two sets of finite measure. 
Writing $u^\lambda$ as a Fourier transform as in \eqref{exp976} and noticing
that $|\gamma_\lambda|=1$, we may 
apply the higher dimensional extension of Nazarov's Uncertainty Principle \cite{J}
to obtain
\begin{equation}
\norm{u_0^\lambda}_{L^2(\R^{2n})}^2
\leq\kappa e^{\kappa |S||J_\lambda\Sigma|}\bigl(\norm{u_0^\lambda}_{L^2(\R^{2n}\setminus S)}^2
+\norm{\ff[\gamma_\lambda u_0^\lambda]}_{L^2(\R^{2n}\setminus J_\lambda\Sigma)}^2\bigr).
\label{exp977}
\end{equation}
A simple computation shows that
$$
\|\ff[\gamma_\lambda u_0^\lambda]\|_{L^2(\R^{2n}\setminus J_\lambda\Sigma)}^2 
=c\|u^\lambda(\cdot,s_0)\|_{L^2(\R^{2n}\setminus \Sigma)}^2
$$
for some constant $c$ independent of $\lambda$ and $s_0$.
On the other hand 
$$
|J_\lambda\Sigma|=|\Sigma|\left(\frac{\lambda}{2}\right)^{2n}\bigl(\cot^2(\lambda s_0)+1\bigr)^{n},
$$
so that \eqref{exp977} reads
\begin{equation}\label{exp975}
\begin{aligned}
\norm{u_0^\lambda}_{L^2(\R^{2n})}^2
\leq\kappa &e^{\kappa \lambda^{2n}\bigl(\cot^2(\lambda s_0)+1\bigr)^{n}|S||\Sigma|}\\
&\Bigl(\norm{u_0^\lambda}_{L^2(\R^{2n}\setminus S)}^2
+\norm{u^\lambda(\cdot,s_0)}_{L^2(\R^{2n}\setminus \Sigma)}^2\Bigr).
\end{aligned}
\end{equation}

Unfortunately, due to the term $\lambda^{2n}\bigl(\cot^2(\lambda s_0)+1\bigr)^{n}$, one can not use Parseval
in the central variable to obtain an estimate of
$\norm{u_0}_{L^2(\heis)}^2$ in terms of $\norm{u_0}_{L^2((\R^{2n}\setminus S)\times\R)}^2$
and $\norm{u(\cdot,s_0)}_{L^2((\R^{2n}\setminus \Sigma)\times\R)}^2$.

There is however one exception: assume there is an $a>0$ such that $u_0^\lambda=0$ for $|\lambda|>a$
and that $s_0a<\pi$. First note that, for $|\lambda|<a$,
\begin{eqnarray*}
\kappa\lambda^{2n}\bigl(\cot^2(\lambda s_0)+1\bigr)^{n}&=&
\kappa\left(\frac{1}{s_0^2}\left(\frac{\sin (\lambda s_0)}{\lambda s_0}\right)^{-2}\cos^2 (\lambda s_0)+\lambda^2\right)^n\\
&\leq&\kappa':= \kappa\frac{2^na^{2n}}{\sin^{2n} (s_0a)}.
\end{eqnarray*}
But then
$$
\begin{aligned}
\|u_0\|_{L^2(\heis)}^2&=\int_{\R}\norm{u_0^\lambda}_{L^2(\R^{2n})}^2\,\mbox{d}\lambda
=\int_{-a}^a\norm{u_0^\lambda}_{L^2(\R^{2n})}^2\,\mbox{d}\lambda\\
\leq&\kappa e^{\kappa'|S||\Sigma|}
\int_{-a}^a\left(\norm{u_0^\lambda}_{L^2(\R^{2n}\setminus S)}^2
+\norm{u^\lambda(\cdot,s_0)}_{L^2(\R^{2n}\setminus \Sigma)}^2\right)\,\mbox{d}\lambda\\
\leq&\kappa e^{\kappa'|S||\Sigma|}
\int_{\R}\left(\norm{u_0^\lambda}_{L^2(\R^{2n}\setminus S)}^2
+\norm{u^\lambda(\cdot,s_0)}_{L^2(\R^{2n}\setminus \Sigma)}^2\right)\,\mbox{d}\lambda\\
=&\kappa e^{\kappa'|S||\Sigma|}
\left(\norm{u_0}_{L^2((\R^{2n}\setminus S)\times\R)}^2
+\norm{u(\cdot,s_0)}_{L^2((\R^{2n}\setminus \Sigma)\times\R)}^2\right).
\end{aligned}
$$
We have thus proved the following:

\begin{proposition}
Let $S,\Sigma\subset\R^{2n}$ be two sets of finite measure. Let $a,s_0>0$ be such
that $s_0a<\pi$. Let $u_0\in L^2(\heis)$ be such that $u_0^\lambda=0$ for $|\lambda|>a$.
Let $u$ be the solution of $i\partial_su+\lll u=0$ with initial condition $u_0$.
Then
$$
\|u_0\|_{L^2(\heis)}^2\leq
\kappa e^{\kappa'|S||\Sigma|}
\left(\norm{u_0}_{L^2((\R^{2n}\setminus S)\times\R)}^2
+\norm{u(\cdot,s_0)}_{L^2((\R^{2n}\setminus \Sigma)\times\R)}^2\right)
$$
where $\kappa'$ depend on $n,a,s_0$ only and $\kappa$ on $n$ only.
\end{proposition}

\begin{remark}[Observability inequality]
Fix $\lambda\in\mathbb R\setminus \{0\},$ and consider the $\lambda$-twisted Laplacian or the special
Hermite operator on $\mathbb R^{2n},$
\[
L_\lambda=\Delta_{\R^{2n}}-\dfrac{|\lambda|^2}{4}(|x|^2+|y|^2)+i\lambda \sum_{j=1}^n \left(
x_j\dfrac{\partial}{\partial y_j}-y_j\dfrac{\partial}{\partial x_j}\right)
\]
Then the solution of the Schr\"odinger equation
\begin{equation}\label{exp938}
\begin{aligned}
i\partial_s u^\lambda(x,y,s)+L_\lambda u^\lambda(x,y,s)=0,\\
 u^\lambda(x,y,0)=u_0^\lambda(x,y) \in L^2(\mathbb R^{2n}),
\end{aligned}
\end{equation}
is obtained in \eqref{exp976} as
\begin{equation*}
u^\lambda(x,y,s)=c_{\lambda,s}\gamma_\lambda(x,y)\ff[\gamma_\lambda u_0^\lambda]
\bigl(J_\lambda(x,y)\bigr)
\end{equation*}
for $s\notin\frac{\pi}{\lambda}\mathbb N.$ Hence, if $S,\Sigma\subset \R^{2n}$ be such that, for every invertible
linear transform $T$, $(S,T\Sigma)$ is a strong annihilating pair, we can obtain the following observability
inequality at two time points for $L_\lambda$ ({\it see} \cite{WWZ} for this kind of inequalities
for the standard Laplacian and its applications).

Let $s_2>s_1\geq 0$ with $s_2-s_1\notin \frac{\pi}{\lambda}\mathbb N.$ Then there exits a constant
$C=C(A,\Sigma,\lambda,s_1,s_2)>0$ such that for all solution $u^\lambda$ of \eqref{exp938},
\begin{equation}
\norm{u_0^\lambda}_{L^2(\R^{2n})}^2
\leq C \Bigl(\norm{u^\lambda(\cdot,s_1)}_{L^2(\R^{2n}\setminus S)}^2
+\norm{u^\lambda(\cdot,s_2)}_{L^2(\R^{2n}\setminus \Sigma)}^2\Bigr).
\end{equation}
Note that such sets $(S, \Sigma)$ are described in Example \ref{examp1}, and the constant $C$ can be
calculated explicitly for some instances.
\end{remark}

\subsection{Limitations to Dynamical Uncertainty Principles}

In this section, we will show that there are limitations
to Dynamical Uncertainty Principles for the Schr\"odinger equation \eqref{exp971}.
This is based on the following example built on an observation from \cite{BGX}.

Let $\ffi\in\ss(\R)$ 
and set
\begin{equation}
\label{eq:counterx}
u_0(z,t)=\int_\R e^{|\lambda|(it-|z|^2/4)}\ffi(\lambda)\,\mbox{d}\lambda.
\end{equation}
A simple computation shows that
$$
\lll u_0=-n\int_\R
|\lambda|e^{|\lambda|(it-|z|^2/4)}\ffi(\lambda)\,\mbox{d}\lambda
=in\partial_tu_0.
$$
It follows that
$$
u(z,t,s)=u_0(z,t-ns)
$$
is a solution of $\lll u(z,t,s)+i\partial_s u(z,t,s)=0$ with $u(z,t,0)=u_0(z,t)$.

Further, as $\ffi\in \ss(\R)$ 
we can write
$$
u_0(z,t)=\ff[e^{\lambda|z|^2/4}\ffi(\lambda)\mathbf{1}_{\R^-}(\lambda)](t)
+\ff[e^{-\lambda|z|^2/4}\ffi(\lambda)\mathbf{1}_{\R^+}(\lambda)](-t)
$$
where $\ff$ is the Fourier transform on $L^1(\R)$.
From Parseval and Fubini, we get
\begin{equation}
\label{eq:parsevlce}
\int_{\heis}|u_0(z,t)|^2\,\mbox{d}z\,\mbox{d}t
=(2\pi)^{n+1}\int_0^{+\infty}
\frac{|\ffi(\lambda)+\ffi(-\lambda)|^2}{\lambda^n}\,\mbox{d}\lambda.
\end{equation}

From now on, we will take $\alpha>0$ and further assume that $\ffi$ is 
supported in $[\alpha,\alpha+1]$. From the definition of $u_0$, we immediately obtain that,
for every $s>0$ 
$$
|u(z,t,s)|\leq  \norm{\ffi}_{L^1(\R)}e^{-\alpha|z|^2/4}.
$$
Further, integrating by parts, we have for every $k\in\N$,
$$
u(z,t,s)=
\int_{\R}\frac{\bigl(-\sign(\lambda)\bigr)^k\ffi^{(k)}(\lambda)}{(it-|z|^2/4-ins)^k}e^{|\lambda|(-|z|^2/4+it-ins)}\,\mbox{d}\lambda
$$
and, as $|e^{|\lambda|(-|z|^2/4+it-ins)}|\leq e^{-\alpha|z|^2/4}$ on $\supp\ffi$,
we obtain
$$
|u(z,t,s)|\leq  \norm{\ffi^{(k)}}_{L^1(\R)}\frac{e^{-\alpha|z|^2/4}}{\bigl((t-ns)^2+|z|^4/16\bigr)^{k/2}}.
$$
This shows that, in Theorem \ref{th:BST}, one can not take $\delta=0$, that is, in the dynamical version
of Hardy's Uncertainty Principle on the Heisenberg group, one needs
to impose a faster than polynomial decay in the central variable $t$.

Next, notice that the function defined in \eqref{eq:counterx} extends into an holomorphic function
in the $z$ variable. More precisely, here $z=(x,y)\in\R^{2n}$ and $u$
extends holomorphically to the cone $\Omega=\{(x+ix',y+iy'), x,x',y,y'\in\R^n\,: |x'|^2+|y'|^2<|x|^2+|y|^2\}$.
In particular, the function $u$ has full support in the $z$-variable.
This explains why the situation is a bit better for Theorem \ref{th:ben} where
no restriction in the central variable is required.

However, this example also shows that there are limitations to quantitative counterparts to Theorem \ref{th:ben}.
Indeed, from $\supp\ffi=[\alpha,\alpha+1]$, we obtain
$$
\int_{\heis}|u(z,t,s)|^2\,\mbox{d}z\,\mbox{d}t\approx \alpha^{-n}\|\ffi\|^2_{L^2(\R)}.
$$

On the other hand, 
$$
\begin{aligned}
\int_{(\R^{2n}\setminus [-r,r]^{2n})\times\R}&|u(z,t,s)|^2\,\mbox{d}z\,\mbox{d}t\\
\leq&\int_{\R^{2n}\setminus [-r,r]^{2n}}\int_\R
\frac{\norm{\ffi'}_{L^1(\R)}^2}{(t-ns)^2+r^4/16}\,\mbox{d}t\,e^{-\alpha|z|^2/2}\,\mbox{d}z\\
\leq&\frac{4\pi}{r^2}\norm{\ffi'}_{L^1(\R)}^2
\left(2\int_r^{+\infty}e^{-\alpha t^2/2}\,\mbox{d}t\right)^{2n}\\
\leq&\frac{2^{2(n+1)}\pi}{\alpha^{2n} r^{2(n+1)}}e^{-n\alpha r^2}\norm{\ffi'}_{L^1(\R)}^2.
\end{aligned}
$$
This shows that for every $s_1<s_2$, the inequality
\begin{multline*}
\int_{\heis}|u_0(z,t)|^2\,\mbox{d}z\,\mbox{d}t\leq C\left(
\int_{\bigl(\R^{2n}\setminus Q(0,r)\bigr)\times\R}|u(z,t,s_1)|^2\,\mbox{d}z\,\mbox{d}t\right.\\
+\left.\int_{\bigl(\R^{2n}\setminus Q(0,r)\bigr)\times\R}|u(z,t,s_2)|^2\,\mbox{d}z\,\mbox{d}t
\right)
\end{multline*}
can not hold for every $u$ solution of $i\partial_s u(z,t,s)+\mathcal{L}u(z,t,s)=0$
with initial condition $u(z,t,0)=u_0(z,t)$, $u_0\in L^2(\heis)$.

\section{Paley-Wiener type theorem}

In this section, we will prove Theorem \ref{th921} from the introduction.
We proceed through a few steps.

\smallskip

\noindent{\bf Step 1.} {\em Reduction to the magnetic Laplacian (twisted Laplacian) on $\mathbb R^{2n}$.}

\smallskip

For $\lambda\in\mathbb R\setminus\{0\},$ let $u^\lambda(z,s)=\int_\mathbb R u(z,t,s)e^{i\lambda t} dt$
be the inverse Fourier transform of $u$ in the central variable. Identify $z$
with $(x,y)=(x_1,\ldots,x_n,y_1,\ldots,y_n)\in\mathbb R^{2n}$.

\begin{proposition}\label{prop923}
Let $u$ be a solution of \eqref{exp921} and satisfies the hypothesis of Theorem \ref{th921}. Then
there is a set $E\subset \mathbb R\setminus\{0\}$ of measure $0$, such that, if
$\lambda\in \mathbb R\setminus(E\cup \{0\})$, then 
 $u^\lambda$ solves
\begin{equation}\label{exp926}
i\partial_s u^\lambda(x,y,s)+\Delta_{C_\lambda}u^\lambda(x,y,s)+V(x,y,s)u^\lambda(x,y,s)=0,
\end{equation}
where $\Delta_{C_\lambda}=\left(\nabla-iC_\lambda\right)^2$ with 
$\dst C_\lambda=\dfrac{\lambda}{2}J\begin{pmatrix}x\\ y\end{pmatrix}$
and
$J=\dst\begin{pmatrix}
0 & I_n \\
-I_n & 0
\end{pmatrix}.$ Moreover,
\begin{eqnarray}
&&\sup_{0\leq s\leq 1}\int_{\mathbb R^{2n}}|u^\lambda(x,y,s)|^2\,\mathrm{d}x\,\mathrm{d}y<\infty, \label{exp927} \\
&&\int_{\mathbb R^{2n}} e^{2a_0 x_1}|u^\lambda(x,y,0)|^2\,\mathrm{d}x\,\mathrm{d}y<\infty \quad \text{for some } a_0>0, \text{ and} \label{exp928} \\
&&\supp\left(u^\lambda(\cdot,\cdot,1)\right)\subseteq
\left\{(x,y)\in\mathbb R^{2n}:|\cos(\lambda)x_1 + \sin(\lambda)y_1|\leq a_3\right\} \label{exp929}
\end{eqnarray}
for some $a_3<\infty$. In addition, for each $(x,y)$, $u^\lambda(x,y,1)$ can be extended holomorpically
in $\lambda$ to a neighbourhood of the real line.
\end{proposition}

\begin{proof}
A direct calculation shows the first part, see, for instance, \cite{FJP}.

For the remaining part, first note the Plancherel formula
\[
\int_{\mathbb R\setminus\{0\}}\int_{\mathbb R^{2n}} |u^\lambda(x,y,s)|^2\,\mathrm{d}x\,\mathrm{d}y\,\mathrm{d}\lambda
=\int_{\mathbb H^n} |u(z,t,s)|^2\,\mathrm{d}z\,\mathrm{d}t,
\]
where $0\leq s\leq 1.$ Hence, it follows from \eqref{exp923} and \eqref{exp924} that
\begin{eqnarray*}
&&\sup_{0\leq s\leq 1}\int_{\mathbb R^{2n}}|u^\lambda(x,y,s)|^2\,\mathrm{d}x\,\mathrm{d}y<\infty, \quad \text{and} \\
&&\int_{\mathbb R^{2n}} e^{2a_0 x_1}|u^\lambda(x,y,0)|^2\,\mathrm{d}x\,\mathrm{d}y <\infty \quad \text{for some } a_0>0,
\end{eqnarray*}
for a.e. $\lambda.$ In addition, condition \eqref{exp925} gives
\begin{eqnarray*}
\supp\left(u^\lambda(\cdot,\cdot,1)\right)\subseteq \left\{x\in\mathbb R^{2n}:|x_1|+|y_1|\leq a_3\right\} \\
\subseteq \left\{x\in\mathbb R^{2n}:|\cos(\lambda)x_1 + \sin(\lambda)y_1|\leq a_3\right\}
\end{eqnarray*}
for some $a_3<\infty.$
\end{proof}

Now onwards $0<|\lambda|<\frac{\pi}{2}$ is fixed.

\smallskip

\noindent{\bf Step 2.} {\em Reduction to the Hermite operator on $\mathbb R^{2n}$.}

\smallskip

\begin{proposition}\label{prop921}
Suppose $u$ and $V$ satisfy the hypothesis of Theorem \ref{th921}. Let $u^\lambda$ be as in Proposition \ref{prop923}.
Write
\begin{equation}\label{exp930}
v(x,y,s)=u^\lambda\left(e^{-s\lambda J}(x,y),s\right).
\end{equation}
Then $v\in C\left([0,1]:L^2(\mathbb R^{2n})\right)$ and solves
\begin{multline}\label{exp931}
i\partial_s v(x,y,s)+\Delta v(x,y,s)-\dfrac{|\lambda|^2}{4}(|x|^2+|y|^2) v(x,y,s)\\+\tilde{V}(x,y,s)v(x,y,s)=0,
\end{multline}
where
\begin{equation}\label{eq:defvtile}
\tilde{V}(x,y,s)=V(e^{-s\lambda J}(x,y),s).
\end{equation}
Additionally, the potential $\tilde{V}$ is bounded and satisfies
\begin{equation}\label{exp932}
\lim_{\rho\rightarrow\infty}\int_0^1\|\tilde{V}(\cdot,\cdot,s)\|_{L^\infty(\mathbb R^{2n}\setminus B_\rho)}
\,\mathrm{d}s=0.
\end{equation}
Furthermore,
\begin{eqnarray}
&&\sup_{0\leq s\leq 1}\int_{\mathbb R^{2n}}|v(x,y,s)|^2\,\mathrm{d}x\,\mathrm{d}y <\infty, \label{exp933} \\
&&\int_{\mathbb R^{2n}} e^{2a_0 x_1}|v(x,y,0)|^2\,\mathrm{d}x\,\mathrm{d}y <\infty \quad \text{for some } a_0>0, \text{ and} \label{exp934} \\
&&\supp\left(v(\cdot,\cdot,1)\right)\subseteq \left\{(x,y)\in\mathbb R^{2n}:|x_1|\leq a_3\right\} \label{exp935}
\end{eqnarray}
for some $a_3<\infty.$
\end{proposition}

\begin{proof}
The fact that $v$ solves \eqref{exp931} follows from \cite[Proposition 5.1]{CF}.

For the remaining of the proof, we will use that
\[e^{-t J}=\begin{pmatrix}
\cos(t) \, I_n & -\sin(t) \, I_n \\
\sin(t) \, I_n & \cos(t) \, I_n
\end{pmatrix}.\]
Next, the boundedness of $\tilde{V}$ is immediate from the boundedness of $V.$ Further, for $s\in [0,1],$
\[\sup_{|(x,y)|\geq \rho}|\tilde{V}(x,y,s)|=\sup_{|(x,y)|\geq \rho}|V(e^{-s\lambda J}(x,y),s)|
=\sup_{|(x,y)|\geq \rho}|V(x,y,s)|.\]
Here the last equality holds as $\left(e^{-s\lambda J}\right)^t=e^{-s\lambda J^t}=e^{s\lambda J}$ we
have $|e^{-s\lambda J}(x,y)|=|(x,y)|.$ Therefore \eqref{exp932} follows from \eqref{exp922}.

Since $u$ satisfies \eqref{exp923}-\eqref{exp925}, Proposition \ref{prop923} ensures that $u^\lambda$
satisfies \eqref{exp927}-\eqref{exp929}. Now, for $s\in [0,1],$ we know that
\[\det(e^{s\lambda J})=e^{s\lambda \text{Tr}(J)}=1.\]
Hence, from \eqref{exp927} we get
\begin{eqnarray*}
\sup_{0\leq s\leq 1}\int_{\mathbb R^{2n}}|v(x,y,s)|^2 \,\mathrm{d}x\,\mathrm{d}y
&=&\sup_{0\leq s\leq 1}\int_{\mathbb R^{2n}}|u^\lambda(e^{-N_\lambda s}(x,y),s)|^2 \,\mathrm{d}x\,\mathrm{d}y \\
&=&\sup_{0\leq s\leq 1}\int_{\mathbb R^{2n}}|u^\lambda(x,y,s)|^2 \,\mathrm{d}x\,\mathrm{d}y <\infty.
\end{eqnarray*}
Further \eqref{exp934} is identical with \eqref{exp928}. To see condition \eqref{exp935},
write $(\tilde{x},\tilde{y})=e^{-\lambda J}(x,y),$ then
\begin{eqnarray*}
&\left|\cos(\lambda) \tilde{x}_1+\sin(\lambda)\tilde{y}_1\right| \\
=&|\cos(\lambda)(\cos(\lambda)x_1-\sin(\lambda)y_1)+\sin(\lambda)(\sin(\lambda)x_1+\cos(\lambda)y_1)|=|x_1|.
\end{eqnarray*}
Thus from \eqref{exp929} and \eqref{exp930} we get
\begin{equation*}
\text{supp}\left(v(\cdot,\cdot,1)\right)\subseteq \left\{(x,y)\in\mathbb R^{2n}:|x_1|\leq a_3\right\}
\end{equation*}
as claimed.
\end{proof}

\smallskip

\noindent{\bf Step 3.} {\em Reduction to the Laplacian on $\mathbb R^{2n}.$}

\smallskip

\begin{proposition}\label{prop922}
Suppose $u$ and $V$ satisfy the hypothesis of Theorem \ref{th921} and let $\tilde V$
be defined in \eqref{eq:defvtile} and, for $z=(x,y)\in\R^{2n}$, define
$$
W(z,s)=(1+|\lambda|^2 s^2)^{-1} \tilde{V}\left(\frac{z}{\sqrt{1+|\lambda|^2 s^2}},
\frac{\arctan(|\lambda|s)}{|\lambda|}\right)
$$
so that $W$ is bounded with
\begin{equation}\label{exp940}
\lim_{\rho\rightarrow\infty}\int_0^{\frac{1}{|\lambda|}\tan(|\lambda|)}
\|W(\cdot,\cdot,s)\|_{L^\infty(\mathbb R^{2n}\setminus B_\rho)}\,\mathrm{d}s=0.
\end{equation}

Let $v$ as in \eqref{exp930}. For
$0<|\lambda|<\frac{\pi}{2},$ set
\begin{equation}\label{exp936}
w(z,s)=\dfrac{\exp\left(i\frac{ s|z|^2}{4(1+|\lambda|^2 s^2)}|\lambda|^2\right)}{(1+|\lambda|^2 s^2)^{n/2}} 
v\left(\frac{z}{\sqrt{1+|\lambda|^2 s^2}},\frac{\arctan(|\lambda|s)}{|\lambda|}\right).
\end{equation}

Then $w\in C\left([0,\frac{1}{|\lambda|}\tan(|\lambda|)]:L^2(\mathbb R^{2n})\right)$ and solves
\begin{equation}\label{exp937}
i\partial_s w(x,y,s)+\Delta w(x,y,s)+W(x,y,s)w(x,y,s)=0,
\end{equation}

Moreover,
\begin{eqnarray}
&&\sup_{0\leq s\leq \frac{1}{|\lambda|}\tan(|\lambda|)}\int_{\mathbb R^{2n}}|w(x,y,s)|^2 \,\mathrm{d}x\,\mathrm{d}y
<\infty, \label{exp941} \\
&&\int_{\mathbb R^{2n}} e^{2a_0 x_1}|w(x,y,0)|^2 \,\mathrm{d}x\,\mathrm{d}y<\infty \quad \text{for some } a_0>0, \text{ and} \label{exp942} \\
&&\supp\left(w\left(\cdot,\cdot,\frac{1}{|\lambda|}\tan(|\lambda|)\right)\right)
\subseteq \left\{(x,y)\in\mathbb R^{2n}:|x_1|\leq a_4\right\} \label{exp943}
\end{eqnarray}
for some $a_4<\infty.$
\end{proposition}

\begin{proof}
First, since $u$ and $V$ satisfy the hypothesis of Theorem \ref{th921}, from Proposition \ref{prop921} we get that
$v$ and $\tilde{V}$ satisfy \eqref{exp932}-\eqref{exp935}. Now,
$$
\begin{aligned}
|W(x,&y,s)|\\
&=(1+|\lambda|^2 s^2)^{-1}\left|\tilde{V}\left(\frac{1}{\sqrt{1+|\lambda|^2 s^2}}(x,y),
\frac{1}{|\lambda|}\arctan(|\lambda|s)\right)\right| \\
&\leq \|\tilde{V}\|_{L^\infty(\mathbb R^{2n} \times [0,1])}
\end{aligned}
$$
for $0\leq s\leq \frac{1}{|\lambda|}\tan(|\lambda|).$
Therefore $\|W\|_{L^\infty(\mathbb R^{2n} \times [0,\frac{1}{|\lambda|}\tan(|\lambda|)])}<\infty.$

Put $(\tilde{x},\tilde{y})=\frac{1}{\sqrt{1+|\lambda|^2 s^2}}(x,y).$ Then $|(x,y)|\geq \rho$ implies
$|(\tilde{x},\tilde{y})|\geq \rho \cos(|\lambda|).$ Hence \eqref{exp932} yields
\begin{multline*}
\lim_{\rho\rightarrow\infty}\int_0^{\frac{1}{|\lambda|}\tan(|\lambda|)}
\|W(\cdot,\cdot,s)\|_{L^\infty(\mathbb R^{2n}\setminus B_\rho)}\,\mbox{d}s\\
\leq\lim_{\rho\rightarrow\infty}\int_0^1
\|\tilde{V}(\cdot,\cdot,s)\|_{L^\infty(\mathbb R^{2n}\setminus B_\rho)}\,\mbox{d}s=0.
\end{multline*}

\smallskip

Next, an argument as in \cite[Proposition 4.3]{CF} shows that $w$ solves \eqref{exp937}.

Further, for $0\leq s\leq \frac{1}{|\lambda|}\tan(|\lambda|),$
$$
\begin{aligned}
\int_{\mathbb R^{2n}}&|w(x,y,s)|^2\,\mbox{d}x\,\mbox{d}y \\
=& \int_{\mathbb R^{2n}}
\left|v\left(\frac{1}{\sqrt{1+|\lambda|^2 s^2}}(x,y),\frac{1}{|\lambda|}\arctan(|\lambda|s)\right)\right|^2 
\,\frac{\mbox{d}x\,\mbox{d}y}{(1+|\lambda|^2 s^2)^{n}} \\
=&\int_{\mathbb R^{2n}}\left|v\left(x,y,\frac{1}{|\lambda|}\arctan(|\lambda|s)\right)\right|^2\,\mbox{d}x\,\mbox{d}y\\
\leq& \sup_{0\leq s\leq 1}\int_{\mathbb R^{2n}}|v(x,y,s)|^2 \,\mbox{d}x\,\mbox{d}y.
\end{aligned}
$$
Thus \eqref{exp941} follows from \eqref{exp933}. The condition \eqref{exp942} is same as \eqref{exp934}.
Finally \eqref{exp943} is easy from \eqref{exp935} with $a_4=\sec(|\lambda|) a_3.$
\end{proof}

\smallskip

\noindent{\bf Step 4.} {\em Conclusion.}

\smallskip

Suppose that $u$ and $V$ satisfy the hypothesis of Theorem \ref{th921}. Then Proposition \ref{prop922} yields $w,$
defined as in \eqref{exp936}, is a solution of \eqref{exp937} and has the properties listed therein.
Therefore, the Paley-Wiener theorem for Schr\"{o}dinger evolutions on the Euclidean space gives $w=0,$ see \cite{KPV}.
This infers $v=0,$ i.e., $u^\lambda=0$ for $0<|\lambda|<\frac{\pi}{2}$.

But, for each $(x,y)\in\mathbb R^{2n}$, $u^\lambda(x,y,1)$ has an
homomorphic extension in $\lambda$ around the real line. We thus conclude that $u^\lambda(x,y,1)=0$ for
all $\lambda\in\mathbb R\setminus \{0\}$. This finally implies that $u=0$.

\section{Data availability}
No data has been generated or analysed during this study.

\section{Funding and/or Conflicts of interests/Competing interests}

The authors have no relevant financial or non-financial interests to disclose.
%
%
%

\end{document}